\documentclass{amsart}
\usepackage{amsfonts,amssymb,amsmath,amsthm}
\usepackage{url}
\usepackage{enumerate}

\newtheorem{thrm}{Theorem}[section]
\newtheorem{lem}[thrm]{Lemma}

\theoremstyle{definition}

\numberwithin{equation}{section}

\author{Wen-ming Lu and Lin Zhang}
\address{
Wen-ming Lu\newline\indent School of Science\\Hangzhou Dianzi
University\\Hangzhou\\People's Republic of
China}\email{lu\_wenming@163.com}
\address{
Lin Zhang\newline\indent Department of Mathematics\\
Zhejiang University\\Hangzhou\\
People's Republic of China}
\email{godyalin@163.com;linyz@zju.edu.cn}

\begin{document}

\title[Direct and Inverse Estimates for Combinations]{Pointwise Approximation Theorems for Combinations of Bernstein Polynomials with inner Singularities}
\begin{abstract}
We give direct and inverse theorems  for the weighted approximation
of functions with inner singularities by combinations of
Bernstein polynomials.
\end{abstract}
\maketitle
\section{Introduction} \label{sect1}
The set of all continuous functions, defined on the interval $I$, is
denoted by $C(I)$. For any $f\in C([0,1])$, the corresponding
\emph{Bernstein operators} are defined as follows:
$$B_n(f,x):=\sum_{k=0}^nf(\frac{k}{n})p_{n,k}(x),$$
where
$$p_{n,k}(x):={n \choose k}x^k(1-x)^{n-k}, \ k=0,1,2,\ldots,n, \ x\in[0,1].$$
Approximation properties of Bernstein operators have been studied
very well (see \cite{Berens}, \cite{Della},
\cite{Totik}-\cite{Lorentz}, \cite{Yu}-\cite{X. Zhou}, for example).
In order to approximate the functions with singularities, Della
Vecchia et al. \cite{Della} and Yu-Zhao \cite{Yu} introduced some
kinds of modified Bernstein operators. Throughout the paper,
$C$ denotes a positive constant independent of $n$ and $x$, which
may be different in different cases.
Ditzian and Totik
\cite{Totik} extended the method of combinations and defined the
following combinations of Bernstein operators:
\begin{align*}
B_{n,r}(f,x):=\sum_{i=0}^{r-1}C_{i}(n)B_{n_i}(f,x),
\end{align*}
with the conditions:
\begin{enumerate}[(a)]
\item $n=n_0<n_1< \cdots <n_{r-1}\leqslant
Cn,$\\
\item $\sum_{i=0}^{r-1}|C_{i}(n)|\leqslant C,$\\
\item
$\sum_{i=0}^{r-1}C_{i}(n)=1,$\\
\item $\sum_{i=0}^{r-1}C_{i}(n)n_{i}^{-k}=0$,\ for $k=1,\ldots,r-1$.
\end{enumerate}
For any positive integer $r$, we consider the determinant
\begin{eqnarray*}
A_{r}:=
\begin {matrix}
\begin{vmatrix}
1 & 1 & 1 & \cdots & 1 \\
2r+1 & 2r+2 & 2r+3 & \cdots & 4r+1 \\
(2r)(2r+1) & (2r+1)(2r+2) & (2r+2)(2r+3) & \cdots & (4r)(4r+1) \\
\cdots & \cdots & \cdots & \ddots & \cdots \\
2\cdots(2r+1) & 3\cdots(2r+2) & 4\cdots(2r+3) & \cdots &
(2r+2)\cdots(4r+1) \end{vmatrix} &
\end{matrix}.
\end{eqnarray*}
We obtain $A_{r}=\prod_{j=2}^{2r}j!$. Thus, there is a
unique solution for the system of nonhomogeneous linear equations:
\begin{align}
\left\{
   \begin{array}{ccccccccc}
   a_1 & + & a_2 & + & \cdots & + & a_{2r+1} & = &1, \\
   (2r+1)a_1 & + & (2r+2)a_2 & + & \cdots & + & (4r+1)a_{2r+1} & = &0, \\
   (2r+1)(2r)a_1 & + & (2r+1)(2r+2)a_2 & + & \cdots & + & (4r)(4r+1)a_{2r+1} & = &0, \\
   &&&\vdots&&&  \\
   (2r+1)!a_1 & + & 3 \cdots (2r+2)a_2 & + & \cdots & + & (2r+2) \cdots (4r+1)a_{2r+1} & =
   &0.
   \end{array} \right.\label{s1}
\end{align}
Let
\begin{eqnarray*}
\psi(x)=\left\{
\begin{array}{lrr}
a_1x^{2r+1}+a_2x^{2r+2}+\cdots+a_{2r+1}x^{4r+1}, &&0<x<1, \\
0,   &&x \leqslant0,  \\
1,  &&x=1.
             \end{array}
\right.
\end{eqnarray*}
with the coefficients $a_1,$ $a_2,$ $\cdots,$ $a_{2r+1}$ satisfying
(\ref{s1}). From (\ref{s1}), we see that $\psi(x)\in
C^{(2r)}(-\infty,+\infty)$, $0\leqslant\psi(x)\leqslant1$ for $0
\leqslant x\leqslant1$. Moreover, it holds that $\psi(1)=1$,\
$\psi^{(i)}(0)=0,\ i=0,1,\cdots,2r$ and $\psi^{(i)}(1)=0,\
i=1,2,\cdots,2r$. \\
Let $$H(f,x):=\sum_{i=1}^{r+1}f(x_{i})l_{i}(x),$$ and
$$l_{i}(x):=\frac{\prod_{j=1,j\neq i}^{r+1}(x-x_{j})}{\prod_{j=1,j\neq i}^{r+1}(x_{i}-x_{j})},\ x_{i}=\frac{[n\xi-({(r-1)/2}+i)]}{n},\ i=1,2, \cdots r+1.$$
Further, let
$$x_{1}^{'}=\frac{[n\xi-2\sqrt{n}]}{n},\ x_{2}^{'}=\frac{[n\xi-\sqrt{n}]}{n},\ x_{3}^{'}=\frac{[n\xi+\sqrt{n}]}{n},\ x_{4}^{'}=\frac{[n\xi+2\sqrt{n}]}{n},$$
and
$${\bar{\psi}}_{1}(x)=\psi(\frac{x-x_{1}^{'}}{x_{2}^{'}-x_{1}^{'}}),\ {\bar{\psi}}_{2}(x)=\psi(\frac{x-x_{3}^{'}}{x_{4}^{'}-x_{3}^{'}}).$$
Set
$${\bar{F}}_{n}(f,x):={\bar{F}}_{n}(x)=f(x)(1-{\bar{\psi}}_{1}(x)+{\bar{\psi}}_{2}(x))+{\bar{\psi}}_{1}(x)(1-{\bar{\psi}}_{2}(x))H(x).$$
We have
\begin{eqnarray*}
{\bar{F}}_{n}(f,x)=\left\{\begin{array}{lr}
f(x),          &       x\in [0,x_{r-5/2}]\cup [x_{r+3/2},1],   \\
f(x)(1-{\bar{\psi}}_{1}(x))+{\bar{\psi}}_{1}(x)
H(x),      &
x\in [x_{r-5/2},x_{r-3/2}],  \\
H(x),          &       x\in [x_{r-3/2},x_{r+1/2}],  \\
H(x)(1-{\bar{\psi}}_{2}(x))+{\bar{\psi}}_{2}(x)f(x), & x\in
[x_{r+1/2},x_{r+3/2}].
            \end{array}
\right.
\end{eqnarray*}
Obviously, ${\bar{F}}_{n}(f,x)$ is linear, reproduces polynomials of
degree $r$, and ${\bar{F}}_{n}(f,x)\in C^{(2r)}([0,1])$, provided
that $f \in C^{(2r)}([0,1]).$
Now, we can define our new combinations of Bernstein operators as
follows:
\begin{eqnarray}
{\bar{B}}_{n,r}(f,x):=B_{n,r}({\bar{F}_{n}},x)=\sum_{i=0}^{r-1}C_{i}(n)B_{n_i}({\bar{F}_{n}},x),\label{s2}
\end{eqnarray}
where $C_{i}(n)$ satisfy the conditions (a)-(d).
\section{The main results} \label{ns}
Let $\phi:[0,1] \longrightarrow R,\ \phi \neq 0$ be an admissible
step-weight function of the Ditzian-Totik modulus of smoothness,
that is, $\phi$ satisfies the
following conditions:\\
\begin{enumerate}[(I)]
\item For every proper subinterval $[a,b] \subseteq [0,1]$ there
exists a constant $C_1 \equiv C(a,b)>0$ such that $C_{1}^{-1}
\leqslant
\phi(x) \leqslant C_1$ for $x \in [a,b].$\\
\item There are two
numbers $\beta(0) \geqslant 0$ and $\beta(1) \geqslant 0$ for which
\begin{align*}
\phi(x)\thicksim \left\{
\begin{array}{lrr}
x^{\beta(0)},&&\ \mbox{as}\ x \rightarrow 0+,    \\
(1-x)^{\beta(1)},&&\ \mbox{as}\ x\rightarrow 1-.
              \end{array}
\right.\\
\end{align*}
($X \sim Y$ means $C^{-1}Y \leqslant X \leqslant CY \mbox{\ for\
some}\ C$).
\end{enumerate}
Combining conditions (\textbf{I}) and (\textbf{II}) on $\phi,$ we can deduce that\\
\begin{align*}
C^{-1}\phi_2(x) \leqslant \phi(x) \leqslant C\phi_2(x),\ x \in
[0,1],
\end{align*}
where $\phi_2(x) = x^{\beta(0)}(1-x)^{\beta(1)}.$\\ Let
$\bar{w}(x)=|x-\xi|^\alpha,\ 0<\xi<1,\ \alpha>0$ and $C_{\bar{w}}:=
\{{f \in C([0,1] \setminus {\xi}) :
\lim\limits_{x\longrightarrow\xi}(\bar{w}f)(x)=0 }\}.$ The norm in
$C_{\bar{w}}$ is defined as
$\|f\|_{C_{\bar{w}}}:=\|{\bar{w}}f\|=\sup\limits_{0\leqslant
x\leqslant 1}|({\bar{w}f})(x)|$. Define
\begin{align*}
W_{\phi}^{r}:= \{f \in C_{\bar{w}}:f^{(r-1)} \in A.C.((0,1)),\
\|\bar{w}\phi^{r}f^{(r)}\|<\infty\},\\
W_{\varphi,\lambda}^{r}:= \{f \in C_{\bar{w}}:f^{(r-1)} \in
A.C.((0,1)),\ \|\bar{w}\varphi^{r\lambda}f^{(r)}\|<\infty\}.
\end{align*}
For $f \in C_{\bar{w}}$, we define the \emph{weighted modulus of
smoothness} by
\begin{align*}
\omega_\phi^{r}(f,t)_{\bar{w}}:=\sup_{0<h\leqslant t}\sup_{0
\leqslant x \leqslant 1}|{\bar{w}}(x)\triangle_{h\phi(x)}^{r}f(x)|,
\end{align*}
where
\begin{align*}
\Delta_{h\phi}^{r}f(x)=\sum_{k=0}^{r}(-1)^{k}{r \choose
k}f(x+(\frac r2-k)h\phi(x)),\\
\overrightarrow{\Delta}_{h}^{r}f(x)=\sum_{k=0}^{r}(-1)^{k}{r
\choose k}f(x+(r-k)h).
\end{align*}
Recently Felten showed the following two theorems in
\cite{Felten}:\\~\\
\textbf{Theorem A.} Let $\varphi(x)=\sqrt{x(1-x)}$ and let
$\phi:[0,1] \longrightarrow R,\ \phi \neq 0$ be an admissible
step-weight function of the Ditzian-Totik modulus of
smoothness(\cite{Totik}) such that $\phi^2$ and $\varphi^2/\phi^2$
are concave. Then, for $f \in C[0,1]$ and $0< \alpha <2,\
|B_n(f,x)-f(x)| \leqslant
\omega_\phi^{2}(f,n^{-1/2}{\frac {\varphi(x)}{\phi(x)}}).$\\~\\
\textbf{Theorem B.} Let $\varphi(x)=\sqrt{x(1-x)}$ and let
$\phi:[0,1] \longrightarrow R,\ \phi \neq 0$ be an admissible
step-weight function of the Ditzian-Totik modulus of smoothness such
that $\phi^2$ and $\varphi^2/\phi^2$ are concave. Then, for $f \in
C[0,1]$ and $0< \alpha <2,\ |B_n(f,x)-f(x)|=O((n^{-1/2}{\frac
{\varphi(x)}{\phi(x)}})^\alpha)$ implies
$\omega_\phi^{2}(f,t)=O(t^\alpha).$\\
Our main results are the following:
\begin{thrm}\label{t1}
For any $\alpha>0,\ \min\{\beta(0),\beta(1)\} \geqslant {\frac 12},\
f \in C_{\bar{w}},$ we have
\begin{align}
|\bar{w}(x)\phi^{r}(x){\bar{B}}_{n,r-1}^{(r)}(f,x)|\leqslant
Cn^{\frac r2}\|\bar{w}f\|.\label{s3}
\end{align}
\end{thrm}
\begin{thrm}\label{t2} For any $\alpha>0,\ f\in W_\phi^{r},$ we have
\begin{align}
|\bar{w}(x)\phi^{r}(x){\bar{B}}_{n,r-1}^{(r)}(f,x)|\leqslant
C\|\bar{w}\phi^{r}f^{(r)}\|.\label{s4}
\end{align}
\end{thrm}
\begin{thrm}\label{t3} For $f\in C_{\bar{w}},\ \alpha>0,\ \min\{\beta(0),\beta(1)\} \geqslant
{\frac 12},\ \alpha_0 \in (0,r),$ we have
\begin{align}
\bar{w}(x)|f(x)-{\bar{B}}_{n,r-1}(f,x)|=O((n^{-{\frac
12}}\phi^{-1}(x)\delta_n(x))^{\alpha_0}) \Longleftrightarrow
\omega_{\phi}^r(f,t)_{\bar{w}}=O(t^{\alpha_0}).\label{s5}
\end{align}
\end{thrm}
\section{Lemmas}
\begin{lem}(\cite{Zhou}) For any non-negative real $u$ and $v$, we
have
\begin{align}
\sum_{k=1}^{n-1}({\frac kn})^{-u}(1-{\frac
kn})^{-v}p_{n,k}(x)\leqslant Cx^{-u}(1-x)^{-v}.\label{s6}
\end{align}
\end{lem}
\begin{lem} (\cite{Della}) If $\gamma \in R,$ then
\begin{align}
\sum_{k=0}^n|k-nx|^\gamma p_{n,k}(x) \leqslant Cn^{\frac
\gamma2}\varphi^\gamma(x).\label{s7}
\end{align}
\end{lem}
\begin{lem} For any $f\in W_\phi^{r},$  $\alpha
>0$, we have
\begin{align}
\|\bar{w}\phi^{r}\bar{F}_{n}^{(r)}\| \leqslant
C\|\bar{w}\phi^{r}f^{(r)}\|.\label{s8}
\end{align}
\end{lem}
\begin{proof} We first prove $x\in [x_{r-5/2},x_{r-3/2}]$ (The same
as the others), we have
\begin{eqnarray*}
|{\bar{w}}(x)\phi^{r}(x){\bar{F}_n}^{(r)}(x)| &\leqslant&
|{\bar{w}}(x)\phi^{r}(x)f^{(r)}(x)| +
|{\bar{w}}(x)\phi^{r}(x)(f(x)-{\bar{F}_n}(x))^{(r)}|\\
&:=& I_1 + I_2.
\end{eqnarray*}
Obviously
\begin{eqnarray*}
I_1 \leqslant C\|{\bar{w}}\phi^{r}f^{(r)}\|.
\end{eqnarray*}
For $I_2,$ we have
\begin{eqnarray*}
I_2 =
{\bar{w}}(x)\phi^{r}(x)|(f(x)-{\bar{F}_n}(x))^{(r)}|={\bar{w}}(x)\phi^{r}(x)\sum_{i=0}^rn^{\frac
i2}|(f(x)-{\bar{F}_n}(x))^{(r-i)}|.
\end{eqnarray*}
By \cite{Totik}, we have
\begin{eqnarray*}
|(f(x)-{\bar{F}_n}(x))^{(r-i)}|_{[x_{r-5/2},x_{r-3/2}]} \leqslant
C(n^{\frac {r-i}{2}}\|f-H\|_{[x_{r-5/2},x_{r-3/2}]} + n^{-\frac
i2}\|f^{(r)}\|_{[x_{r-5/2},x_{r-3/2}]}).
\end{eqnarray*}
So
\begin{eqnarray*}
I_2 &\leqslant& Cn^{\frac
r2}{\bar{w}}(x)\phi^{r}(x)\|f-H\|_{[x_{r-5/2},x_{r-3/2}]} +
C{\bar{w}}(x)\phi^{r}(x)\|f^{(r)}\|_{[x_{r-5/2},x_{r-3/2}]}\\
&:=& T_1 + T_2.
\end{eqnarray*}
By Taylor expansion, we have
\begin{eqnarray}
f({x_i})=\sum_{u=0}^{r-1}\frac{(x_i-x)^u}{u!}f^{(u)}(x)+{\frac
1{(r-1)!}}\int_{x}^{x_{i}}(x_i-s)^{r-1}f^{(r)}(s)ds,\label{s9}
\end{eqnarray}
It follows from (\ref{s9}) and the identity
\begin{eqnarray*}
\sum\limits_{i=1}^{r}x_{i}^{v}l_{i}(x)=Cx^v,\ v=0,1,\cdots,r.
\end{eqnarray*}
we have
\begin{eqnarray*}
H(f,x)&=&\sum_{i=1}^{r}\sum_{u=0}^{r}\frac{(x_i-x)^u}{u!}f^{(u)}(x)l_{i}(x)+{\frac
1{(r-1)!}}\sum_{i=1}^{r}l_{i}(x)\int_{x}^{x_{i}}(x_i-s)^{r-1}f^{(r)}(s)ds\nonumber\\
&=&f(x)+C\sum_{u=1}^{r}f^{(u)}(x)(\sum_{v=0}^{u}C_{u}^{v}(-x)^{u-v}\sum_{i=1}^{r}x_{i}^{v}l_{i}(x))\nonumber\\
&&+{\frac
1{(r-1)!}}\sum_{i=1}^{r}l_{i}(x)\int_{x}^{x_{i}}(x_i-s)^{r-1}f^{(r)}(s)ds,
\end{eqnarray*}
which implies that
\begin{eqnarray*}
{\bar{w}(x)}\phi^{r}(x)|f(x)-H(f,x)|={\frac
1{(r-1)!}}{\bar{w}(x)}\phi^{r}(x)\sum_{i=1}^{r}l_{i}(x)\int_{x}^{x_{i}}(x_i-s)^{r-1}f^{(r)}(s)ds,
\end{eqnarray*}
since $|l_{i}(x)|\leqslant C$ for $x\in [x_{r-5/2},x_{r-3/2}],\
i=1,2,\cdots,r$. It follows from
$\frac{|x_i-s|^{r-1}}{{\bar{w}}(s)}\leqslant
\frac{|x_i-x|^{r-1}}{{\bar{w}}(x)},$ $s$ between $x_i$ and $x$, then
\begin{eqnarray*}
{\bar{w}(x)}\phi^{r}(x)|f(x)-H(f,x)|&=&C\bar{w}(x)\phi^{r}(x)\sum_{i=1}^{r}\int_{x}^{x_{i}}(x_i-s)^{r-1}|f^{(r)}(s)|ds\nonumber\\
&\leqslant&C\phi^{r}(x)\|{\bar{w}}\phi^{r}f^{(r)}\|\sum_{i=1}^{r}(x_i-x)^{r-1}\int_{x}^{x_{i}}\phi^{-r}(s)ds\nonumber\\
&\leqslant&{\frac {C}{n^{r/2}}}\|{\bar{w}}\phi^{r}f^{(r)}\|.
\end{eqnarray*}
So
\begin{eqnarray*}
I_2 \leqslant C\|{\bar{w}}\phi^{r}f^{(r)}\|.
\end{eqnarray*}
Then, the lemma is proved.
\end{proof}
\begin{lem} If $f\in W_\phi^{r},$  $\alpha>0$, then
\begin{eqnarray}
{\bar{w}(x)}|g(x)-H(g,x)| \leqslant C(\frac
{\delta_n(x)}{\sqrt{n}\phi(x)})^r\|{\bar{w}}\phi^{r}g^{(r)}\|.\label{s10}
\end{eqnarray}
\end{lem}
\begin{proof}  By Taylor expansion, we have
\begin{eqnarray*}
f({x_i})=\sum_{u=0}^{r-1}\frac{(x_i-x)^u}{u!}f^{(u)}(x)+{\frac
1{(r-1)!}}\int_{x}^{x_{i}}(x_i-s)^{r-1}f^{(r)}(s)ds,
\end{eqnarray*}
It follows from the above equality and the identity
\begin{eqnarray*}
\sum\limits_{i=1}^{r}x_{i}^{v}l_{i}(x)=Cx^v,\ v=0,1,\cdots,r.
\end{eqnarray*}
we have
\begin{eqnarray*}
H(f,x)&=&\sum_{i=1}^{r}\sum_{u=0}^{r}\frac{(x_i-x)^u}{u!}f^{(u)}(x)l_{i}(x)+{\frac
1{(r-1)!}}\sum_{i=1}^{r}l_{i}(x)\int_{x}^{x_{i}}(x_i-s)^{r-1}f^{(r)}(s)ds\nonumber\\
&=&f(x)+C\sum_{u=1}^{r}f^{(u)}(x)(\sum_{v=0}^{u}C_{u}^{v}(-x)^{u-v}\sum_{i=1}^{r}x_{i}^{v}l_{i}(x))\nonumber\\
&&+{\frac
1{(r-1)!}}\sum_{i=1}^{r}l_{i}(x)\int_{x}^{x_{i}}(x_i-s)^{r-1}f^{(r)}(s)ds,
\end{eqnarray*}
which implies that
\begin{eqnarray*}
{\bar{w}(x)}|g(x)-H(g,x)|={\frac
1{(r-1)!}}{\bar{w}(x)}\sum_{i=1}^{r}l_{i}(x)\int_{x}^{x_{i}}(x_i-s)^{r-1}g^{(r)}(s)ds,
\end{eqnarray*}
since $|l_{i}(x)|\leqslant C$ for $x\in [x_{r-5/2},x_{r-3/2}],\
i=1,2,\cdots,r$. It follows from
$\frac{|x_i-s|^{r-1}}{{\bar{w}}(s)}\leqslant
\frac{|x_i-x|^{r-1}}{{\bar{w}}(x)},$ $s$ between $x_i$ and $x$, then
\begin{eqnarray*}
{\bar{w}(x)}|g(x)-H(g,x)|&\leqslant&C\bar{w}(x)\sum_{i=1}^{r}\int_{x}^{x_{i}}(x_i-s)^{r-1}|g^{(r)}(s)|ds\nonumber\\
&\leqslant& C\|{\bar{w}}\phi^{r}g^{(r)}\|\sum_{i=1}^{r}(x_i-x)^{r-1}\int_{x}^{x_{i}}\phi^{-r}(s)ds\nonumber\\
&\leqslant& C{\frac
{\varphi^r(x)}{\phi^{r}(x)}}\|{\bar{w}}\phi^{r}g^{(r)}\|\sum_{i=1}^{r}(x_i-x)^{r-1}\int_{x}^{x_{i}}\varphi^{-r}(s)ds\\
&\leqslant&C{\frac
{\delta_n^r(x)}{\phi^{r}(x)}}\|{\bar{w}}\phi^{r}g^{(r)}\|\sum_{i=1}^{r}(x_i-x)^{r-1}\int_{x}^{x_{i}}\varphi^{-r}(s)ds\\
&\leqslant& C(\frac
{\delta_n(x)}{\sqrt{n}\phi(x)})^r\|{\bar{w}}\phi^{r}g^{(r)}\|.
\end{eqnarray*}
\end{proof}
\begin{lem} For any $\alpha >0,$
$f\in C_{\bar{w}},$ we have
\begin{eqnarray}
\|{\bar{w}}{\bar{B}}_{n,r-1}(f)\|\leqslant
C\|{\bar{w}}f\|.\label{s11}
\end{eqnarray}
\end{lem}
\begin{proof} By (\ref{s2}), we have
\begin{eqnarray*}
|{\bar{w}(x)}{\bar{B}}_{n,r}(f,x)|&=&|{\bar{w}(x)}B_{n,r}({\bar{F}_{n}},x)|\leqslant
{\bar{w}(x)}\sum_{i=0}^{r-1}\sum_{k=1}^{n_i-1}C_{i}(n)|{\bar{F}}_{n}{(\frac
k{n_i})}|p_{n_i,k}(x)\nonumber\\
&&+
{\bar{w}(x)}\sum_{i=0}^{r-1}C_{i}(n)|{\bar{F}}_{n}{(0)}|p_{n_i,0}(x)+
{\bar{w}(x)}\sum_{i=0}^{r-1}C_{i}(n)|{\bar{F}}_{n}{(1)}|p_{n_i,n_i}(x)\nonumber\\
&:=&I_1 + I_2 + I_3.
\end{eqnarray*}
Now, the theorem can be proved easily. \end{proof}
\begin{lem}(\cite{Wang}) Let $\min\{\beta(0), \beta(1)\}
\geqslant {\frac 12}$, then for $r \in N,\ 0<t<{\frac {1}{8r}}$ and
${\frac {rt}{2}} < x < 1-{\frac {rt}{2}},$ we have
\begin{eqnarray}
\int_{-{\frac {t}{2}}}^{\frac {t}{2}} \cdots \int_{-{\frac
{t}{2}}}^{\frac {t}{2}}\phi^{-r}(x+\sum_{k=1}^ru_k)du_1 \cdots du_r
\leqslant Ct^r\phi^{-r}(x).\label{s12}
\end{eqnarray}
\end{lem}
\begin{lem} Let
$A_n(x):={\bar{w}(x)}\sum\limits_{|k-n\xi|\leqslant
\sqrt{n}}p_{n,k}(x)$. Then $A_n(x)\leqslant Cn^{-\alpha/2}$ for
$0<\xi <1$ and $\alpha>0$.
\end{lem}
\begin{proof} If $|x- \xi|\leqslant {\frac {3}{\sqrt{n}} }$, then the
statement is trivial. Hence assume $0 \leqslant x \leqslant
\xi-{\frac {3}{\sqrt{n}} }$ (the case $\xi+{\frac {3}{\sqrt{n}} }
\leqslant x \leqslant 1$ can be treated similarly). Then for a fixed
$x$ the maximum of $p_{n,k}(x)$ is attained for
$k=k_n:=[n\xi-\sqrt{n}]$. By using Stirling's formula, we get
\begin{eqnarray*}
p_{n,k_n}(x)&\leqslant& C{\frac {({\frac
{n}{e}})^n\sqrt{n}x^{k_n}(1-x)^{n-k_n}}{({\frac
{k_n}{e}})^{k_n}\sqrt{k_n}({\frac
{n-k_n}{e}})^{n-k_n}\sqrt{n-k_n}}}\nonumber\\
&\leqslant& {\frac {C}{\sqrt{n}}}({\frac {nx}{k_n}})^{k_n}({\frac
{n(1-x)}{n-k_n}})^{n-k_n}\nonumber\\
&=&{\frac {C}{\sqrt{n}}}(1-{\frac {k_n-nx}{k_n}})^{k_n}(1+{\frac
{k_n-nx}{n-k_n}})^{n-k_n}.
\end{eqnarray*}
Now from the inequalities
$$k_n-nx=[n\xi-\sqrt{n}]-nx>n(\xi-x)-\sqrt{n}-1\geqslant {\frac 12}n(\xi-x),$$
and
$$1-u\leqslant e^{-u-{\frac 12}u^2},\ 1+u\leqslant e^u,\ u\geqslant 0.$$
We have that the second inequality is valid. To prove the first one
we consider the function $\lambda(u)=e^{-u-{\frac 12}u^2}+u-1.$ Here
$\lambda(0)=0,\ \lambda^\prime(u)=-(1+u)e^{-u-{\frac 12}u^2}+1,\
\lambda^\prime(0)=0,\ \lambda^{\prime\prime}(u)=u(u+2)e^{-u-{\frac
12}u^2}\geqslant 0,$ whence $\lambda(u)\geqslant 0$ for $u\geqslant
0$. Hence
\begin{eqnarray*}
p_{n,k_n}(x) &\leqslant&{\frac {C}{\sqrt{n}}}exp\{k_n[-{\frac
{k_n-nx}{k_n}}-{\frac 12}({\frac {k_n-nx}{k_n}})^2] +
k_n-nx\}\nonumber\\
&=&{\frac {C}{\sqrt{n}}}exp\{-\frac {({k_n-nx})^2}{2k_n}\}\leqslant
e^{-Cn(\xi-x)^2}.
\end{eqnarray*}
Thus $A_n(x)\leqslant C(\xi-x)^\alpha e^{-Cn(\xi-x)^2}$. An easy
calculation shows that here the maximum is attained when
$\xi-x={\frac {C}{\sqrt{n}}}$ and the lemma follows.
\end{proof}
\begin{lem} For $0<\xi <1,\ \alpha,\ \beta>0$, we
have
\begin{eqnarray}
{\bar{w}(x)}\sum\limits_{|k-n\xi|\leqslant \sqrt{n}}|k-nx|^\beta
p_{n,k}(x)\leqslant Cn^{{\frac
{\beta-\alpha}{2}}}\varphi^\beta(x).\label{s13}
\end{eqnarray}
\end{lem}
\begin{proof} By (\ref{s7}) and the lemma 3.7, we have
\begin{eqnarray*}
{\bar{w}(x)}^{\frac
{1}{2n}}({\bar{w}(x)\sum\limits_{|k-n\xi|\leqslant
\sqrt{n}}p_{n,k}(x)})^{\frac
{2n-1}{2n}}(\sum\limits_{|k-n\xi|\leqslant \sqrt{n}}|k-nx|^{2n\beta}
p_{n,k}(x))^{\frac {1}{2n}}\leqslant Cn^{{\frac
{\beta-\alpha}{2}}}\varphi^\beta(x).
\end{eqnarray*}
\end{proof}
\begin{lem} For any $\alpha >0,$ $f\in W_{\phi}^{r},$ we have
\begin{align}
\|{\bar{w}}{\bar{B}}_{n,r-1}^{(r)}(f)\|\leqslant
Cn^{r}\|{\bar{w}}f\|.\label{s14}
\end{align}
\end{lem}
\begin{proof} We first prove $x\in [0,{\frac 1n})$ (The same as $x\in
(1-{\frac 1n},1]$), now
\begin{eqnarray*}
|{\bar{w}}(x){\bar{B}}_{n,r-1}^{(r)}(f,x)|&\leqslant&
{\bar{w}}(x)\sum_{i=0}^{r-2}{\frac
{n_{i}!}{({n_{i}-r})!}}\sum_{k=0}^{n_i-r}C_{i}(n)|\overrightarrow{\Delta}_{\frac
1{n_i}}^{r}{\bar{F}}_{n}{(\frac k{n_i})}|p_{n_i-r,k}(x)\nonumber\\
&\leqslant&
C{\bar{w}}(x)\sum_{i=0}^{r-2}n_{i}^{r}\sum_{k=0}^{n_i-r}|\overrightarrow{\Delta}_{\frac
1{n_i}}^{r}{\bar{F}}_{n}{(\frac k{n_i})}|p_{n_i-r,k}(x)\nonumber\\
&\leqslant&
C{\bar{w}}(x)\sum_{i=0}^{r-2}n_{i}^{r}\sum_{k=0}^{n_i-r}\sum_{j=0}^{r}C_{r}^{j}|{\bar{F}}_{n}({\frac
{k+r-j}{n_i}})|p_{n_i-r,k}(x)\nonumber\\
&\leqslant&
C{\bar{w}}(x)\sum_{i=0}^{r-2}n_{i}^{r}\sum_{j=0}^{r}C_{r}^{j}|{\bar{F}}_{n}({\frac
{r-j}{n_i}})|p_{n_i-r,0}(x)\nonumber\\
&&+\
C{\bar{w}}(x)\sum_{i=0}^{r-2}n_{i}^{r}\sum_{j=0}^{r}C_{r}^{j}|{\bar{F}}_{n}({\frac
{n_{i}-j}{n_i}})|p_{n_i-r,n_i-r}(x)\nonumber\\
&&+\
C{\bar{w}}(x)\sum_{i=0}^{r-2}n_{i}^{r}\sum_{k=1}^{n_i-r-1}\sum_{j=0}^{r}C_{r}^{j}|{\bar{F}}_{n}({\frac
{k+r-j}{n_i}})|p_{n_i-r,k}(x)\nonumber\\
&:=&H_1 +H_2 + H_3.
\end{eqnarray*}
We have
\begin{eqnarray*}
H_1&\leqslant&
C{\bar{w}}(x)\sum_{i=0}^{r-2}n_{i}^{r}\sum_{j=0}^{r}|\bar{F}_{n}({\frac
{r-j}{n_i}})|p_{n_i-r,0}(x)\nonumber\\
&\leqslant&
Cn^{r}\|{\bar{w}}f\|\sum_{i=0}^{r-2}\sum_{j=0}^{r}(\frac{n_{i}|x-\xi|}{r-j-n_{i}\xi})^{\alpha}(1-x)^{{n_{i}}-r}\nonumber\\
&\leqslant&
Cn^{r}\|{\bar{w}}f\|\sum_{i=0}^{r-2}(n_{i}|x-\xi|)^{\alpha}(1-x)^{{n_{i}}-r}\nonumber\\
&\leqslant& Cn^{r}\|{\bar{w}}f\|.
\end{eqnarray*}
Similarly, we can get $H_2\leqslant Cn^{r}\|{\bar{w}}f\|,$ and
$H_3\leqslant
Cn^{r}\|{\bar{w}}f\|$. \\~\\
When $x\in [{\frac 1n},1-{\frac 1n}],$ according to \cite{Totik}, we
have
\begin{eqnarray*}
&&|{\bar{w}}(x){\bar{B}}_{n,r-1}^{(r)}(f,x)|\nonumber\\
&=&|{\bar{w}}(x)B_{n,r-1}^{(r)}({\bar{F}_{n}},x)|\nonumber\\
&\leqslant&{\bar{w}}(x)(\varphi^{2}(x))^{-r}\sum_{i=0}^{r-2}\sum_{j=0}^{r}Q_{j}(x,n_i)C_{i}(n)n_{i}^{j}\sum_{k/n_i\in
A}|(x-{\frac kn_{i}})^{j}||{\bar{F}}_{n}({\frac kn_{i}})|p_{n_i,k}(x)\nonumber\\
&&+\
{\bar{w}}(x)(\varphi^{2}(x))^{-r}\sum_{i=0}^{r-2}\sum_{j=0}^{r}Q_{j}(x,n_i)C_{i}(n)n_{i}^{j}\sum_{x_2^{\prime}
\leqslant k/n_i\leqslant x_3^\prime}|(x-{\frac kn_{i}})^{j}||H({\frac
kn_{i}})|p_{n_i,k}(x)\nonumber\\
&:=&\sigma_1+ \sigma_2.
\end{eqnarray*}
Where $A:=[0,x_2^{\prime}]\cup [x_3^{\prime},1]$, $H$ is a linear
function. If ${\frac kn_{i}}\in A,$ when ${\frac
{\bar{w}(x)}{\bar{w}(\frac {k}{n_{i}})}}\leqslant C(1+n_i^{-{\frac
{\alpha}{2}}}|k-n_ix|^\alpha),$ we have $|k-n_{i}\xi|\geqslant
{\frac {\sqrt{n_{i}}}{2}}$, also $Q_{j}(x,n_i)=(n_ix(1-x))^{[{\frac
{r-j}{2}}]},$ and
$(\varphi^{2}(x))^{-r}Q_{j}(x,n_i)n_{i}^{j}\leqslant
C(n_i/\varphi^{2}(x))^{\frac {r+j}{2}}.$ By (\ref{s7}), then
\begin{eqnarray*}
\sigma_1&\leqslant&
C{\bar{w}}(x)\sum_{i=0}^{r-2}\sum_{j=0}^{r}(\frac
{n_{i}}{\varphi^{2}(x)})^{\frac
{r+j}{2}}\sum_{k=0}^{n_{i}}|(x-{\frac
kn_{i}})^{j}||{\bar{F}}_{n}({\frac kn_{i}})|p_{n_i,k}(x)\nonumber\\
&\leqslant& C\|{\bar{w}}f\|\sum_{i=0}^{r-2}\sum_{j=0}^{r}(\frac
{n_{i}}{\varphi^{2}(x)})^{\frac
{r+j}{2}}\sum_{k=0}^{n_{i}}[1+n_i^{-{\frac
{\alpha}{2}}}|k-n_ix|^\alpha]|x-{\frac kn_{i}|^{j}}p_{n_i,k}(x)\nonumber\\
&:=&I_1 + I_2.
\end{eqnarray*}
By a simple calculation, we have $I_1\leqslant
Cn^{r}\|{\bar{w}}f\|$. By (\ref{s7}), then
\begin{eqnarray*}
I_2\leqslant
C\|{\bar{w}}f\|\sum_{i=0}^{r-2}\sum_{j=0}^{r}n_i^{-({{\frac
{\alpha}{2}}}+j)}(\frac {n_{i}} {\varphi^{2}(x)})^{\frac
{r+j}{2}}\sum_{k=0}^{n_{i}}|k-n_ix|^{\alpha+j}p_{n_i,k}(x)\leqslant
Cn^{r}\|{\bar{w}}f\|.
\end{eqnarray*}
We note that $|H({\frac {k}{n_i}})|\leqslant
max(|H(x_1^\prime)|,|H(x_4^\prime)|):=H(a)$.\\~\\
If $x\in [x_1^\prime,x_4^\prime],$ we have ${\bar{w}(x)}\leqslant
{\bar{w}(a)}.$ So, if $x\in [x_1^\prime,x_4^\prime],$ then
$$\sigma_2\leqslant Cn^r{\bar{w}(a)}H(a)\leqslant Cn^{r}\|\bar{w}f\|.$$
If $x\notin [x_1^\prime,x_4^\prime],$ then
${\bar{w}(a)}>n_i^{-{\frac {\alpha}{2}}},$ by (\ref{s13}), we have
\begin{eqnarray*}
\sigma_2\leqslant
C{\bar{w}(a)}H(a){\bar{w}}(x)\sum_{i=0}^{r-2}\sum_{j=0}^{r}n_{i}^{\frac
{\alpha}{2}}(\frac {n_{i}}{\varphi^{2}(x)})^{\frac {r+j}{2}}
\sum_{x_2^{\prime} \leqslant k/n_i\leqslant x_3^\prime}|x-{\frac
kn_{i}}|^{j}p_{n_i,k}(x)\\ \leqslant Cn^{r}\|{\bar{w}}f\|.
\end{eqnarray*}
It follows from combining the above inequalities that the lemma is
proved.
\end{proof}
\section{Proof of Theorems}
\subsection{Proof of Theorem \ref{t1}}

When $f\in C_{\bar{w}},\ \min\{\beta(0), \beta(1)\}
\geqslant {\frac 12},$ we discuss it as follows:\\~\\
\textit{Case 1.} If $0\leqslant \varphi(x)\leqslant {\frac
{1}{\sqrt{n}}}$, by (\ref{s14}), we have
\begin{align}
|\bar{w}(x)\phi^{r}(x){\bar{B}}_{n,r-1}^{(r)}(f,x)| =
C\varphi^r(x)\cdot{\frac
{\phi^{r}(x)}{\varphi^r(x)}}|\bar{w}(x){\bar{B}}_{n,r-1}^{(r)}(f,x)|\nonumber\\
\leqslant Cn^{\frac r2}\|\bar{w}f\|.\label{s15}
\end{align} \textit{Case 2.} \textit{Case 2.} If
$\varphi(x)> {\frac {1}{\sqrt{n}}}$, we have
\begin{eqnarray*}
&&|{\bar{B}}_{n,r-1}^{(r)}(f,x)|=|B_{n,r-1}^{(r)}({\bar{F}_{n}},x)|\nonumber\\
&\leqslant&(\varphi^{2}(x))^{-r}\sum_{i=0}^{r-2}\sum_{j=0}^{r}Q_{j}(x,n_i)C_{i}(n)n_{i}^{j}\sum_{k=0}^{n_i}|(x-{\frac
kn_{i}})^{j}||{\bar{F}}_{n}({\frac kn_{i}})|p_{n_i,k}(x),
\end{eqnarray*}
where\\
$Q_{j}(x,n_i)=(n_ix(1-x))^{[{\frac {r-j}{2}}]},$ and
$(\varphi^{2}(x))^{-r}Q_{j}(x,n_i)n_{i}^{j}\leqslant
C(n_i/\varphi^{2}(x))^{\frac {r+j}{2}}$.\\
So
\begin{eqnarray}
&&|{\bar{w}(x)}\phi^{r}(x){\bar{B}}_{n,r-1}^{(r)}(f,x)|\nonumber\\
&\leqslant&
C{\bar{w}(x)}\phi^{r}(x)\sum_{i=0}^{r-2}\sum_{j=0}^{r}({\frac
{n_{i}}{\varphi^2(x)}})^{\frac {r+j}{2}}\sum_{k=0}^{n_i}|(x-{\frac
kn_{i}})^{j}||{\bar{F}}_{n}({\frac kn_{i}})|p_{n_i,k}(x)\nonumber\\
&=& C{\bar{w}(x)}\phi^{r}(x)\sum_{i=0}^{r-2}\sum_{j=0}^{r}({\frac
{n_{i}}{\varphi^2(x)}})^{\frac {r+j}{2}}\sum_{k/n_i\in A}|(x-{\frac
kn_{i}})^{j}||{\bar{F}}_{n}({\frac kn_{i}})|p_{n_i,k}(x)\nonumber\\
&&+ C{\bar{w}(x)}\phi^{r}(x)\sum_{i=0}^{r-2}\sum_{j=0}^{r}({\frac
{n_{i}}{\varphi^2(x)}})^{\frac {r+j}{2}}\sum_{x_2^{\prime} \leqslant
k/n_i\leqslant x_3^\prime}|{(x-{\frac kn_{i}})^{j}}||H({\frac
kn_{i}})|p_{n_i,k}(x)\nonumber\\
&:=&\sigma_1+ \sigma_2.\label{s16}
\end{eqnarray}
Where $A:=[0,x_2^{\prime}]\cup [x_3^{\prime},1],$ we can easily get
$\sigma_1\leqslant Cn^{\frac r2}\|{\bar{w}}f\|,$ and
$\sigma_2\leqslant Cn^{\frac r2}\|{\bar{w}}f\|.$ By bringing these
facts together, the theorem is proved. $\Box$
\subsection{Proof of Theorem \ref{t2}}
When $f\in W_\phi^{r},$ by \cite{Totik}, we have
\begin{align}
B_{n,r-1}^{(r)}(\bar{F}_{n},x)=\sum_{i=0}^{r-2}C_{i}(n)n_{i}^{r}\sum_{k=0}^{n_{i}-r}\overrightarrow{\Delta}_{\frac
1{n_{i}}}^{r}\bar{F}_{n}({\frac kn_{i}})p_{n_i-r,k}(x).\label{s17}
\end{align}
If $0<k<n_{i}-r,$ we have
\begin{align}
|\overrightarrow{\Delta}_{\frac 1{n_{i}}}^{r}\bar{F}_{n}({\frac
kn_{i}}) |\leqslant Cn_{i}^{-r+1}\int_{0}^{\frac
{r}{n_{i}}}|\bar{F}_{n}^{(r)}({\frac kn_{i}}+u)|du,\label{s18}
\end{align}
If $k=0,$ we have
\begin{align}
|\overrightarrow{\Delta}_{\frac
1{n_{i}}}^{r}\bar{F}_{n}(0)|\leqslant C\int_{0}^{\frac
{r}{n_{i}}}u^{r-1}|\bar{F}_{n}^{(r)}(u)|du,\label{s19}
\end{align}
Similarly
\begin{align}
|\overrightarrow{\Delta}_{\frac 1{n_{i}}}^{r}\bar{F}_{n}({\frac
{n_{i}-r}{n_{i}}})|\leqslant Cn_i^{-r+1}\int_{1-{\frac
{r}{n_{i}}}}^{1}(1-u)^{\frac r2}|\bar{F}_{n}^{(r)}(u)|du.\label{s20}
\end{align}
By (\ref{s17}), we have
\begin{align}
|\bar{w}(x)\phi^{r}(x)\bar{B}_{n,r-1}^{(r)}(f,x)|\nonumber\\
\leqslant
C{\bar{w}(x)}\phi^{r}(x)\sum_{i=0}^{r-2}n_{i}^{r}\sum_{k=0}^{n_{i}-r}|\overrightarrow{\Delta}_{\frac
1{n_{i}}}^{r}\bar{F}_{n}({\frac kn_{i}})|p_{n_i-r,k}(x)\nonumber\\
\leqslant
C{\bar{w}(x)}\phi^{r}(x)\sum_{i=0}^{r-2}n_{i}^{r}\sum_{k=1}^{n_{i}-r-1}|\overrightarrow{\Delta}_{\frac
1{n_{i}}}^{r}\bar{F}_{n}({\frac kn_{i}})|p_{n_i-r,k}(x)\nonumber\\
+
C{\bar{w}(x)}\phi^{r}(x)\sum_{i=0}^{r-2}n_{i}^{r}|\overrightarrow{\Delta}_{\frac
1{n_{i}}}^{r}\bar{F}_{n}(0)|p_{n_i-r,0}(x)\nonumber\\ +
C{\bar{w}(x)}\phi^{r}(x)\sum_{i=0}^{r-2}n_{i}^{r}|\overrightarrow{\Delta}_{\frac
1{n_{i}}}^{r}\bar{F}_{n}(1)|p_{n_i-r,n_i-r}(x). \label{s21}
\end{align}
which combining with (\ref{s18})-(\ref{s20}) give
\begin{align*}
|\bar{w}(x)\phi^{r}(x)\bar{B}_{n,r-1}^{(r)}(f,x)|\leqslant
C\|w\phi^{r}f^{(r)}\|.\Box
\end{align*}
\\
Combining with the theorem \ref{t1} and theorem \ref{t2}, we can
obtain \\~\\
\textbf{Corollary} {\it For any $\alpha
>0,\ 0 \leqslant \lambda \leqslant 1,$ we have
\begin{align}
|\bar{w}(x)\varphi^{r\lambda}(x)\bar{B}^{(r)}_{n,r-1}(f,x)|\leqslant \left\{
\begin{array}{lrr}
Cn^{r/2}{\{max\{n^{r(1-\lambda)/2},\varphi^{r(\lambda-1)}(x)\}\}}\|\bar{w}f\|,\ &&    f\in C_{\bar{w}},    \\
C\|\bar{w}\varphi^{r\lambda}f^{(r)}\|,\ && f\in
W_{\bar{w},\lambda}^{r}.\label{s22}
              \end{array}
\right.
\end{align}}
\subsection{Proof of Theorem \ref{t3}}
\subsubsection{The direct theorem} We know
\begin{align}
\bar{F}_n(t)=\bar{F}_n(x)+\bar{F}'_n(t)(t-x) + \cdots +
{\frac{1}{(r-1)!}}\int_x^t
(t-u)^{r-1}\bar{F}_n^{(r)}(u)du,\label{s23}\\
B_{n,r-1}((\cdot-x)^k,x)=0, \ k=1,2,\cdots,r-1.\label{s24}
\end{align}
According to the definition of $W_\phi^{r},$ for any $g \in
W_\phi^{r},$ we have
$\bar{B}_{n,r-1}(g,x)=B_{n,r-1}(\bar{G}_{n}(g),x),$ and
$\bar{w}(x)|\bar{G}_{n}(x)-B_{n,r-1}(\bar{G}_{n},x)|=\bar{w}(x)|B_{n,r-1}(R_r(\bar{G}_n,t,x),x)|,$
thereof $R_r(\bar{G}_n,t,x)=\int_x^t
(t-u)^{r-1}\bar{G}^{(r)}_n(u)du,$ we have
\begin{align}
\bar{w}(x)|\bar{G}_{n}(x)-B_{n,r-1}(\bar{G}_{n},x)|  \leqslant
C\|\bar{w}\phi^{r}\bar{G}^{(r)}_n\|\bar{w}(x)B_{n,r-1}({\int_x^t{\frac
{|t-u|^{r-1}}{\bar{w}(u)\phi^{r}(u)}du,x}})\nonumber\\
 \leqslant
C\|\bar{w}\phi^{r}\bar{G}^{(r)}_n\|\bar{w}(x)(B_{n,r-1}(\int_x^t{\frac
{|t-u|^{r-1}}{\phi^{2r}(u)}}du,x))^{\frac 12}\cdot
\nonumber\\
(B_{n,r-1}(\int_x^t{\frac {|t-u|^{r-1}}{\bar{w}^2(u)}}du,x))^{\frac
12}.\label{s25}
\end{align}
also
\begin{align}
\int_x^t{\frac {|t-u|^{r-1}}{\phi^{2r}(u)}}du \leqslant C{\frac
{|t-x|^r}{\phi^{2r}(x)}},\ \int_x^t{\frac
{|t-u|^{r-1}}{\bar{w}^2(u)}}du \leqslant C{\frac
{|t-x|^r}{\bar{w}^2(x)}}.\label{s26}
\end{align}
By (\ref{s7}), (\ref{s8}) and (\ref{s26}), we have
\begin{align}
\bar{w}(x)|\bar{G}_{n}(x)-B_{n,r-1}(\bar{G}_{n},x)| \leqslant
C\|\bar{w}\phi^{r}\bar{G}^{(r)}_n\|\phi^{-r}(x)B_{n,r-1}
(|t-x|^r,x)\nonumber\\
\leqslant Cn^{-\frac r2}{\frac
{\varphi^{r}(x)}{\phi^{r}(x)}}\|\bar{w}\phi^{r}\bar{G}^{(r)}_n\|\nonumber\\
\leqslant Cn^{-\frac r2}{\frac
{\delta_n^r(x)}{\phi^{r}(x)}}\|\bar{w}\phi^{r}\bar{G}^{(r)}_n\|\nonumber\\
= C(\frac
{\delta_n(x)}{\sqrt{n}\phi(x)})^r\|\bar{w}\phi^{r}\bar{G}^{(r)}_n\|.\label{s27}
\end{align}
By (\ref{s8}), (\ref{s10}) and (\ref{s27}), when $g \in W_\phi^{r},$
then
\begin{align}
\bar{w}(x)|g(x)-\bar{B}_{n,r-1}(g,x)| \leqslant
\bar{w}(x)|g(x)-\bar{G}_{n}(g,x)| +
\bar{w}(x)|\bar{G}_{n}(g,x)-\bar{B}_{n,r-1}(g,x)|\nonumber\\
\leqslant {\bar{w}}(x)|g(x)-H(g,x)|\nonumber\\  +\ \ C(\frac
{\delta_n(x)}{\sqrt{n}\phi(x)})^r\|{\bar{w}}\phi^{r}\bar{G}^{(r)}_n\|\nonumber\\
\leqslant C(\frac
{\delta_n(x)}{\sqrt{n}\phi(x)})^r\|{\bar{w}}\phi^{r}g^{(r)}\|.\label{s28}
\end{align}
For $f \in C_{\bar{w}},$ we choose proper $g \in W_\phi^{r},$ by
(\ref{s11}) and (\ref{s28}), then
\begin{align*}
{\bar{w}}(x)|f(x)-{\bar{B}}_{n,r-1}(f,x)| \leqslant
{\bar{w}}(x)|f(x)-g(x)| + {\bar{w}}(x)|{\bar{B}}_{n,r-1}(f-g,x)|\\ +
{\bar{w}}(x)|g(x)-{\bar{B}}_{n,r-1}(g,x)|\\
\leqslant C(\|{\bar{w}}(f-g)\|+(\frac
{\delta_n(x)}{\sqrt{n}\phi(x)})^r\|{\bar{w}}\phi^{r}g^{(r)}\|)\\
\leqslant C\omega_\phi^r(f,\frac
{\delta_n(x)}{\sqrt{n}\phi(x)})_{\bar{w}}
. \Box
\end{align*}
\subsubsection{The inverse theorem}
The \emph{weighted $K$-function} is given by
\begin{align*}
K_{r,\phi}(f,t^r)_{\bar{w}}:=\underset{g}{\inf
}\{\|{\bar{w}}(f-g)\|+t^{r}\|{\bar{w}}\phi^{r}g^{(r)}\|:g \in
W_{\phi}^{r}\}.
\end{align*}
By \cite{Totik}, we have
\begin{align}
C^{-1}\omega_{\phi}^{r}(f,t)_{\bar{w}} \leqslant
K_{r,\phi}(f,t^r)_{\bar{w}} \leqslant
C\omega_{\phi}^{r}(f,t)_{\bar{w}}. \label{s29}
\end{align}
\begin{proof} Let $\delta>0,$ by (\ref{s29}), we choose proper $g$ so
that
\begin{align}
\|{\bar{w}}(f-g)\| \leqslant
C\omega_{\phi}^{r}(f,\delta)_{\bar{w}},\
\|{\bar{w}}\phi^{r}g^{(r)}\| \leqslant
C\delta^{-r}\omega_{\phi}^{r}(f,\delta)_{\bar{w}}.\label{s30}
\end{align}
For $r \in N,\ 0<t<{\frac {1}{8r}}$ and ${\frac {rt}{2}} < x <
1-{\frac {rt}{2}},$ we have
\begin{align}
|\bar{w}(x)\Delta_{h\phi}^rf(x)| \leqslant
|\bar{w}(x)\Delta_{h\phi}^r(f(x)-{\bar{B}}_{n,r-1}(f,x))|+|\bar{w}(x)\Delta_{h\phi}^r{\bar{B}}_{n,r-1}(f-g,x)|\nonumber\\
 +\ |\bar{w}(x)\Delta_{h\phi}^r{\bar{B}}_{n,r-1}(g,x)|\nonumber\\
\leqslant \sum_{j=0}^rC_r^j(n^{-\frac
12}{\frac {\delta_n(x+({\frac r2}-j)h\phi(x))}{\phi(x+({\frac r2}-j)h\phi(x))}})^{\alpha_0}\nonumber\\
 +\ \ \int_{-{\frac {h\phi(x)}{2}}}^{\frac
{h\phi(x)}{2}}\cdots \int_{-{\frac {h\phi(x)}{2}}}^{\frac
{h\phi(x)}{2}}\bar{w}(x){\bar{B}^{(r)}_{n,r-1}(f-g,x+\sum_{k=1}^ru_k)}du_1\cdots
du_r\nonumber\\
 +\ \ \int_{-{\frac {h\phi(x)}{2}}}^{\frac
{h\phi(x)}{2}}\cdots \int_{-{\frac {h\phi(x)}{2}}}^{\frac
{h\phi(x)}{2}}\bar{w}(x){\bar{B}^{(r)}_{n,r-1}(g,x+\sum_{k=1}^ru_k)}du_1\cdots
du_r\nonumber\\
:= J_1+J_2+J_3.\label{s31}
\end{align}
Obviously
\begin{align}
J_1 \leqslant C(n^{-\frac
12}\phi^{-1}(x)\delta_n(x))^{\alpha_0}.\label{s32}
\end{align}
By (\ref{s14}) and (\ref{s30}), we have
\begin{align}
J_2 \leqslant Cn^r\|\bar{w}(f-g)\|\int_{-{\frac
{h\phi(x)}{2}}}^{\frac {h\phi(x)}{2}}\cdots \int_{-{\frac
{h\phi(x)}{2}}}^{\frac
{h\phi(x)}{2}}du_1 \cdots du_r\nonumber\\
\leqslant Cn^rh^r\phi^{r}(x)\|\bar{w}(f-g)\|\nonumber\\
\leqslant
Cn^rh^r\phi^{r}(x)\omega_{\phi}^{r}(f,\delta)_{\bar{w}}.\label{s33}
\end{align}
By the first inequality of (\ref{s22}), we let $\lambda=1,$ and
(\ref{s30}), then
\begin{align}
J_2 \leqslant Cn^{\frac r2}\|\bar{w}(f-g)\|\int_{-{\frac
{h\phi(x)}{2}}}^{\frac {h\phi(x)}{2}} \cdots \int_{-{\frac
{h\phi(x)}{2}}}^{\frac
{h\phi(x)}{2}}\varphi^{-r}(x+\sum_{k=1}^ru_k)du_1 \cdots du_r\nonumber\\
\leqslant Cn^{\frac r2}h^r\phi^{r}(x)\varphi^{-r}(x)\|\bar{w}(f-g)\|\nonumber\\
\leqslant Cn^{\frac
r2}h^r\phi^{r}(x)\varphi^{-r}(x)\omega_{\phi}^{r}(f,\delta)_{\bar{w}}.\label{s34}
\end{align}
By (\ref{s12}) and (\ref{s30}), we have
\begin{align}
J_3 \leqslant C\|\bar{w}\phi^{r}g^{(r)}\|\bar{w}(x)\int_{-{\frac
{h\phi(x)}{2}}}^{\frac {h\phi(x)}{2}} \cdots \int_{-{\frac
{h\phi(x)}{2}}}^{\frac
{h\phi(x)}{2}}{\bar{w}^{-1}(x+\sum_{k=1}^ru_k)}\phi^{-r}(x+\sum_{k=1}^ru_k)du_1 \cdots du_r\nonumber\\
\leqslant Ch^r\|\bar{w}\phi^{r}g^{(r)}\|\nonumber\\
\leqslant
Ch^r\delta^{-r}\omega_{\phi}^{r}(f,\delta)_{\bar{w}}.\label{s35}
\end{align}
Now, by (\ref{s31})-(\ref{s35}), there exists a constant $M>0$ so
that
\begin{align*}
|\bar{w}(x)\Delta_{h\phi}^rf(x)| \leqslant C((n^{-\frac
12}{\frac {\delta_n(x)}{\phi(x)}})^{\alpha_0}\\
+ \min\{n^{\frac r2}{\frac
{\phi^r(x)}{\varphi^r(x)}},n^r\phi^r(x)\}h^r\omega_{\phi}^{r}(f,\delta)_{\bar{w}}
+ h^r\delta^{-r}\omega_{\phi}^{r}(f,\delta)_{\bar{w}})\\
\leqslant C((n^{-\frac
12}{\frac {\delta_n(x)}{\phi(x)}})^{\alpha_0}\\
+\  h^rM^r(n^{-\frac 12}{\frac {\varphi(x)}{\phi(x)}}\ +\ n^{-\frac
12}{\frac
{n^{-1/2}}{\phi(x)}})^{-r}\omega_{\phi}^{r}(f,\delta)_{\bar{w}}
+ h^r\delta^{-r}\omega_{\phi}^{r}(f,\delta)_{\bar{w}})\\
\leqslant C((n^{-\frac
12}{\frac {\delta_n(x)}{\phi(x)}})^{\alpha_0}\\
+ h^rM^r(n^{-\frac 12}{\frac
{\delta_n(x)}{\phi(x)}})^{-r}\omega_{\phi}^{r}(f,\delta)_{\bar{w}}
+ h^r\delta^{-r}\omega_{\phi}^{r}(f,\delta)_{\bar{w}}).\\
\end{align*}
When $n \geqslant 2,$ we have
\begin{align*}
n^{-\frac 12}\delta_n(x) < (n-1)^{-\frac 12}\delta_{n-1}(x)
\leqslant \sqrt{2}n^{-\frac 12}\delta_n(x),
\end{align*}
Choosing proper $x,\ \delta,\  n \in N,$ so that
\begin{align*}
n^{-\frac 12}{\frac {\delta_n(x)}{\phi(x)}} \leqslant \delta <
(n-1)^{-\frac 12}{\frac {\delta_{n-1}(x)}{\phi(x)}},
\end{align*}
Therefore
\begin{align*}
|{\bar{w}}(x)\Delta_{h\phi}^rf(x)| \leqslant C\{\delta^{\alpha_0} +
h^r\delta^{-r}\omega_{\phi}^{r}(f,\delta)_{\bar{w}}\}.
\end{align*}
Which implies
\begin{align*}
\omega_\phi^{r}(f,t)_{\bar{w}} \leqslant C\{\delta^{\alpha_0} +
h^r\delta^{-r}\omega_{\phi}^{r}(f,\delta)_{\bar{w}}\}.\label{s38}
\end{align*}
So, by Berens-Lorentz lemma in \cite{Totik}, we get
\begin{align*}
\omega_\phi^{r}(f,t)_{\bar{w}} \leqslant Ct^{\alpha_0}.
\end{align*}
\end{proof}

\end{document}